\documentclass[11pt]{amsart}
\usepackage{amssymb,amsmath}
\usepackage{latexsym,epsfig,multicol,anysize,enumerate}
\usepackage{latexsym,amsfonts,amssymb,amscd}
\usepackage{fancybox,shadow,color,psfrag,float}
\usepackage{graphicx,epsfig,tikz}
\usepackage{geometry}
\usepackage{verbatim}
\usepackage{mdframed}
\usepackage{lipsum}
\usepackage{geometry}                % See geometry.pdf to learn the layout options. There are lots.
\usepackage{graphicx}
\usepackage{amssymb}
\usepackage{epstopdf}

\newtheorem{teo}{Theorem}[section]
\newtheorem{defi}[teo]{Definition}
\newtheorem{propo}[teo]{Proposition}
\newtheorem{lem}[teo]{Lemma}
\newtheorem{cor}[teo]{Corollary}
\newtheorem{remark}[teo]{Remark}
\newtheorem*{ramif}{Ramification condition}

\numberwithin{equation}{section}

\renewcommand{\char}[1]{char\,#1}
\def\Cal#1{{\mathcal #1}}

\def\HH{\Cal{H}}

\def\P{\mathbb P}
\def\f{\mathbb F}

\def\a{\alpha}
\def\b{\beta}
\def\g{\gamma}
\def\t{\theta}
\def\d{\delta}
\def\ii{\infty}

\newcommand{\tr}{\operatorname{Tr}}

\newcommand{\m}{\mathcal}

\title[A problem of Beelen, Garcia and Stichtenoth]{A problem of Beelen, Garcia and Stichtenoth on an Artin-Schreier tower in characteristic two}
\author[M. Chara]{Mar\'ia Chara}
\address{Instituto de Matem\'atica Aplicada del Litoral\\
   Colectora Ruta Nac. N 168 \\
   Paraje El Pozo\\
   (3000) Santa Fe - Argentina}
\email{mchara@santafe-conicet.gov.ar}

\author[H. Navarro]{Horacio Navarro}
\address{Instituto de Matem\'atica Aplicada del Litoral\\
   Colectora Ruta Nac. N 168 \\
   Paraje El Pozo\\
   (3000) Santa Fe - Argentina}
\email{horacio.navarro@correounivalle.edu.co}

\author[R. Toledano]{Ricardo Toledano}
\address{ Departamento de Matem\'atica\\
   Facultad de Ingenier\'ia Qu\'imica (UNL)\\
   Santiago del Estero 2829\\
   (3000) Santa Fe - Argentina}
\email{rtoledano@santafe-conicet.gov.ar} 
\begin{document}
\maketitle

\begin{abstract}
We study a tower of function fields of Artin-Schreier type  over a
finite field with $2^s$ elements. The study of the asymptotic
behavior of this tower was left as an open problem by Beelen,
Garc\'ia and Stichtenoth in $2006$. We prove that this tower  is
asymptotically good for $s$ even and asymptotically bad for $s$ odd.
\end{abstract}

\section{Introduction}

In 2006  P. Beelen, A. Garcia and H. Stichtenoth gave the first steps in \cite{BGS06} towards the classification, according to their asymptotic behavior, of recursive towers of function fields over a finite field $\f_q$ with $q$ elements.  They focused in recursive towers defined by equations of the form $f(y)=g(x)$, where $f$ and $g$ are suitable rational functions over $\f_q$. Towers defined in this way were called $(f,g)$-towers over $\f_q$. In particular, they noticed that many $(f,g)$-towers can be recursively defined by  equations of the form $h(y)=A\cdot h(B\cdot x)$ for some polynomial $h$ over $\f_q$ and $A$, $B\in GL(2,\f_q)$. Here the symbol $A\cdot u$ stands for the usual action of elements of $GL(2,\f_q)$ as fractional transformations, i.e. \[\left(\begin{array}{cc}
a & b\\ c & d
\end{array}\right)\cdot u\colon\!\!=\frac{au+b}{cu+d}.\] This was a key observation that allowed them to obtain classification results in the important cases of recursive towers of Kummer  and Artin-Schreier type. As an application of these results, they gave a complete list of all $(f,g)$-towers of Artin-Schreier type with $\deg f= \deg g=2$ over the finite field $\f_2$. They checked that all the possible cases were already considered in previous works, except for the following Artin-Schreier tower $\mathcal{H}$ recursively defined by the equation
\begin{equation}\label{ec1}
y^2+y=\frac{x}{x^2+x+1},
\end{equation}
over $\f_2$. Nothing else was said about this tower and in fact they posed, as an open problem,  to determine when the above equation \eqref{ec1} defines an asymptotically good tower over $\f_{2^s}$ with $s\geq 1$. The aim of this work is to give a complete answer to this question.

More precisely we will show that the tower $\mathcal{H}$ defined by \eqref{ec1} has finite genus over $\f_{2^s}$ for every $s\geq 1$ but its splitting rate is zero when $s$ is odd  and positive when $s$ is even. Consequently $\mathcal{H}$ is asymptotically bad in the first case and  asymptotically good in the second one.  In particular we finish, in this way, with the study of all $(f,g)$-towers of Artin-Schreier type with $\deg f= \deg g=2$ over $\f_2$ initiated by  P. Beelen, A. Garcia and H. Stichtenoth in \cite{BGS06}.

The organization of this paper is as follows. In Section \ref{preliminaries} we give some basic definitions and we recall some known results. In Section \ref{genus} we deal with the genus of $\mathcal{H}$ over $\f_{2^s}$ and we prove in Theorem \ref{finitegenus} that the genus of $\mathcal{H}$ is finite for any positive integer $s$. Finally, Section \ref{splitting} is devoted to the study of the splitting rate of $\mathcal{H}$. We prove in Theorems \ref{teo4.2} and \ref{goodf4} that the splitting rate of $\mathcal{H}$ is zero when $s$ is odd and positive when $s$ is even respectively.  Section \ref{splitting} is certainly the most intricate and interesting part of the paper. This is because of the ramification behavior of some rational places which is quite different, when going up in the tower, according to the parity of $s$. What happens is that for $s$ odd, the number of rational places becomes constant after the first step in the tower but for $s$ even the situation changes completely. Roughly speaking  when $s$ is even some rational places, which are totally ramified in the first steps of the tower, start to split completely in the tower. We give a detailed study of this behavior throughout several technical lemmas which heavily rely on the explicit construction of what we call Artin-Schreier elements of type $1$ and $2$ (see Definition \ref{astype}) in each step of the tower $\HH$.  
\section{Preliminaries}\label{preliminaries}
We give now the basic definitions and concepts of function fields and towers of function fields which will be used in this paper. The standard reference for all of this is \cite{stichbook}. Let $k$ be a perfect field. A function field (of one variable) $F$ over $k$ is a finite algebraic extension $F$ of the rational function field $k(x)$, where $x$ is a transcendental element over $k$.

Let $F$ be a function field over $k$. The symbol $\P(F) $ stands for the set of all places of $F$ and $g(F)$ for the genus of $F$.

Let $F'$ be a finite extension of $F$ and let $Q\in\P(F')$.
We will write $Q|P$ when the place $Q$ of $F'$ lies over the place $P$ of $F$, i.e. $P=Q\cap F$. In this case the symbols $e(Q | P)$ and $d(Q | P)$ denote, as usual, the ramification index and the different exponent of $Q | P$, respectively.

 A {\em tower} $\mathcal{F}$ (of function fields) over $k$ is a sequence $\mathcal{F}=\{F_i\}_{i=0}^{\infty}$ of function fields over $k$ such that

\begin{enumerate}[(a)]
\item $F_i \subsetneq F_{i+1}$ for all $i\geq 0.$
\item The extension $F_{i+1}/F_{i}$ is finite and separable, for all $i\geq 1.$
\item The field $k$ is algebraically closed in  $F_i$, for all $i\geq 0.$
\item The genus $g(F_i)\rightarrow\infty$  as  $i\rightarrow\infty$.
\end{enumerate}

A tower $\mathcal{F}=\{F_i\}_{i=0}^{\infty}$ over $k$ is called {\it recursive} if there exist a sequence of transcendental elements $\{x_i\}_{i=0}^{\infty}$ over $k$ and a bivariate polynomial $H(X,Y)\in k[X,Y]$ such that $F_0=k(x_0)$ and
\[F_{i+1}=F(x_i),\]
where $H(x_i,x_{i+1})=0$ for all $i\geq 0$. Associated to any recursive tower $\mathcal{F}$ we have its {\it basic function field} $F=k(x,y)$ where $H(x,y)=0$ and $x$ is a transcendental element over $k$.

The following definitions are important when dealing with the so called asymptotic behavior of a tower (see  \cite[Chapter 7 ]{stichbook} for details) over $k$.  Let $\mathcal{F}=\{F_i\}_{i=0}^{\infty}$ be a  tower of function fields over  $k$. The \textit{the genus} $\gamma(\m{F})$ of $\mathcal{F}$ over $F_0$ is defined as
\[\gamma(\mathcal{F}):=\lim_{i\rightarrow \infty}\frac{g(F_i)}{[F_i:F_0]}\,.\]
When $k=\f_q$ we denote by $N(F_i)$ the number of rational places (i.e., places of degree one) of $F_i$ and the \textit{ splitting rate} $\nu(\m{F})$ of $\mathcal{F}$ over $F_0$ is defined as
\[\nu(\mathcal{F}):=\lim_{i\rightarrow \infty}\frac{N(F_i)}{[F_i:F_0]}.\]

We say that a tower $\mathcal{F}$ is \textit{asymptotically good} over $\f_q$ if $\nu(\mathcal{F})>0$ and $\gamma(\mathcal{F}) < \infty$. Otherwise we say that  $\mathcal{F}$  is  \textit{asymptotically bad}. Equivalently, a tower $\mathcal{F}$ is asymptotically good over  $\f_q$ if and only if \textit{the limit} of the tower $\mathcal{F}$ \[\lambda(\mathcal{F}):=\lim_{i\rightarrow \infty}\frac{N(F_i)}{g(F_i)}=\frac{\nu(\mathcal{F})}{\gamma(\mathcal{F})},\]
is positive.

In the study of the asymptotic behavior of a tower $\mathcal{F}=\{F_i\}_{i=0}^{\infty}$ over $\f_q$, the following sets play an important role: the \textit{ramification locus} $R(\mathcal{F})$ of $\mathcal{F}$, which is the set of   places $P$ of $F_0$ such that $P$  is ramified in $F_i$ for some $i\geq 1$ and the \textit{splitting locus} $Sp(\mathcal{F})$ of $\mathcal{F}$, which is the set of rational places $P$ of $F_0$ such that $P$ splits completely in $F_i$ for all $i\geq 1$.

Let $B\geq 0$ be a real number and let $F'/F$ be a finite extension of function fields over $k$. A place $P$ of $F$ is called \textit{$B$-bounded} in $F'$ if
\[d(Q|P)\leq B\cdot (e(Q|P)-1),\]
for any place $Q$ of $F'$ lying over $P$. The extension $F'/F$ is called {\em $B$-bounded} if every place of $F$ is $B$-bounded in $F'$.  A tower $\{F_i\}_{i=0}^{\infty}$ over $k$ is called {\em $B$-bounded} if every extension $F_i/F_0$ is $B-$bounded.  In  \cite[Proposition 1.5]{galoisclosure} the following result is proved.

\begin{propo}\label{p1} A $B$-bounded tower $\mathcal{F}=\{F_i\}_{i=0}^{\infty}$ over $k$  with finite ramification locus has finite genus. More precisely, the following bound for the genus of $\mathcal{F}$ holds:
\[\gamma(F)\leq g(F_0)-1+\frac{B}{2}\sum_{P\in \mathcal{R}(\mathcal{F})}\deg P.\]
\end{propo}
We immediately see that a tower $\mathcal{F}=\{F_i\}_{i=0}^{\infty}$ over $\f_q$ is asymptotically  good if the tower is $B$-bounded, $\mathcal{R}(\mathcal{F})$ is a finite set and  $Sp(\mathcal{F})\neq \emptyset$.
The next proposition is proved in \cite[Proposition 3.9.6]{stichbook}.
\begin{propo}\label{extensionporconstantes}Let $\mathcal{F}=\{F_i\}_{i=0}^{\infty}$ be  a tower over $\f_q$ and let $\mathcal{E}=\{E_i\}_{i=0}^{\infty}$ be the tower over $\f_{q^n}$ where each $E_i$ is the composite of the fields $F_i$ and $\f_{q^n}$. Then $Sp(\mathcal{F})\subseteq Sp(\mathcal{E})\cap F_0$.
\end{propo}

We recall now from \cite[Section 7.4]{stichbook} the concept of weakly ramified extensions.

\begin{defi}\label{d1} Let $F$ be a function field over  $k$ with $Char(k)=p$. A finite field extension $E/F$ is said to be weakly ramified, if the following conditions hold:
\begin{enumerate}[(i)]
\item There exist intermediate fields $F=E_0\subseteq E_1 \subseteq \dots \subseteq E_n=E$ such that all extensions $E_{i+1}/E_i$ are Galois $p$-extensions (i.e., $[E_{i+1}:E_i]$ is a power of $p$), for $i=0,1\dots, n-1$.
\item For any $P\in \P(F)$ and any $Q\in \P(E)$ lying over $P$, the different exponent is given by
$d(Q|P)=2(e(Q|P)-1)$.
\end{enumerate}

\end{defi}

We have (see  \cite[Remark 7.4.11, Proposition 7.4.13]{stichbook}) the following results:
\begin{propo}\label{r1}  Let $E/F$ be an extension of function fields over $k$ such that $[E:F]= p^m$ where $p=\char F$. Assume that there exist a chain of intermediate fields
\[F=E_0\subseteq E_1 \subseteq \dots \subseteq E_n =E\] with the property
$E_{i+1}/E_i$ is a Galois $p$-extension for all $i=0,1\dots, n-1$.
Let $P\in \P(F)$ and $Q\in \P(E)$ lying over $P$ and let $Q_i$ be the restriction of $Q$ to $E_i$ for $i=0,\dots n-1$. Then the following conditions are equivalent:
\begin{enumerate}[(i)]
\item   $d(Q|P)=2(e(Q|P)-1)$.
\item  $d(Q_{i+1}|Q_i)=2(e(Q_{i+1}|Q_i)-1)$ for  $i=0,\ldots,
n-1$.
\end{enumerate}

\end{propo}

Notice that if every extension $F_i/F_0$ of a tower $\{F_i\}_{i=0}^{\infty}$ over $k$ is  weakly ramified then the tower is $2$-bounded.

\begin{propo}\label{r2}   Let $E/F$ be a finite extension of function fields over  $k$ and let $M$ and $N$ be intermediate fields of $E\supseteq F$ such that $E=MN$ is the compositum of $M$ and $N$. If both extensions $M/F$ and $N/F$ are weakly ramified then $E/F$ is weakly ramified.
\end{propo}

Now we can prove the following result which will be useful in the study of the genus of the tower $\mathcal{H}$.
\begin{cor}\label{wram} Let $\mathcal{F}=\{F_i\}_{i=0}^{\infty}$ be a recursive tower of function fields over $k$. Let $F=k(x,y)$ be the basic function field associated to $F$ and suppose that the extensions $F/k(x)$ and $F/k(y)$ are weakly ramified. Then each extension $F_i/F_0$ is weakly ramified (in which case we say that $\mathcal{F}$ is weakly ramified tower). Furthermore  $\mathcal{F}$ has finite genus if the ramification locus of $\mathcal{F}$ is finite.
\end{cor}
\begin{proof}
We have that $F_0=k(x_0)$ and $F_{i+1}=F(x_i)$ for $i\geq 0$ where
$\{x_i\}_{i=0}^{\infty}$ is a sequence of transcendental elements
over $k$. By hypothesis, the extensions $F_1/k(x_0)$ and
$F_1/k(x_1)$ are weakly ramified. In particular  $F_1/F_0$ is weakly
ramified. Now suppose that $F_i/F_0$ is weakly ramified. Since
$F_i=k(x_0,\ldots,x_i)$, the extensions $F_i/F_0$ and
$k(x_1,\ldots,x_i,x_{i+1})/k(x_1)$ are isomorphic over $k$ (by the map
sending $x_j\rightarrow x_{j+1}$ for $j=0,\ldots,i$), hence
$k(x_1,\ldots,x_i,x_{i+1})/k(x_1)$ is weakly ramified. Since
$F_{i+1}$ is the compositum of the fields $F_1=k(x_0,x_1)$ and
$k(x_1,\ldots,x_i,x_{i+1})$, we have that the extension
$F_{i+1}/k(x_1)$ is weakly ramified by Proposition \ref{r2} and so
$F_{i+1}/F_i$ is a weakly ramified extension by Proposition
\ref{r1}. This immediately implies that $F_{i+1}/F_0$ is a weakly
ramified extension and this proves the first part. In particular we
have that $\mathcal{F}$ is a $2$-bounded tower, so the second statement
follows directly from Proposition \ref{p1}.
\end{proof}

\section{The genus of the tower $\mathcal{H}$}\label{genus}

Let $\mathcal{H}=\{F_i\}_{i=0}^{\infty}$ be the tower of function fields over $k=\f_{2^s}$ recursively defined  by \eqref{ec1}, which is the equation of Artin-Schreier type
\begin{equation}\nonumber
y^2+y=\frac{x}{x^2+x+1}\,.
\end{equation}
In this section we will show that $\mathcal{H}$ is a tower with finite genus for any $s\in \mathbb{N}$.

It is well known that both the different exponent and the ramification index, remain the same under constant field extensions so we can work also  over an algebraic closure $K$ of $k$ when needed.

The next lemma describes the ramification in the basic function field $F=K(x,y)$ corresponding to $\mathcal{H}$ (see Figure~\ref{figu1}). We will use the following notation: the symbol $P_a$ (resp. $R_a$) denotes the only zero of $x+a$ (resp. $y+a$)  in $K(x)$ (resp. $K(y)$) for $a\in K$, while $P_{\infty}$ (resp. $R_{\infty}$) denotes the only pole of $x$ in $K(x)$ (resp. $y$ in $K(y)$). Notice that we can write $x-a$ instead of $x+a$ because we are in even characteristic.

\begin{lem}\label{lema3.1} Let $F=K(x,y)$ be the basic function field corresponding to the tower  $\mathcal{H}$ over $K$. Let $\alpha_1, \alpha_2 \in K$ be the roots of the polynomial $h(T)=T^2+T+1$. We have that the following hold.

\begin{enumerate}[(i)]
\item\label{i} The place $P_0$ of $K(x)$ splits completely in
$F$. Moreover, if $Q\in \P(F)$ lies over  $P_0$,
then $Q\cap K(y)$ is either $R_0$ or $R_1$.

\item\label{ii} The place $P_{\infty}$ of $K(x)$ splits completely in
$F$. Moreover, if $Q\in \P(F)$ lies over $P_{\infty}$
then $Q\cap K(y)$ is either $R_0$ or $R_1$.

\item\label{iii} The place $P_1$ of $K(x)$ splits completely in
$F$. Moreover, if $Q\in \P(F)$ lies over $P_1$
then $Q\cap K(y)$ is either $R_{\alpha_1}$ or $R_{\alpha_2}$.

\item\label{iv} The places $P_{\a_i}$ de $K(x)$ are totally ramified in
$F$ and its  different exponent is $2$. Moreover, if $Q\in \P(F)$ lies over $P_{\a_i}$ for  $i=1$ or $2$, then
$Q\cap K(y)=R_{\infty}$.

\item\label{v} The places $R_{\a_i}$ of $K(y)$ for $i=1,2$ are totally ramified in
$F$ and its  different exponent is $2$.

\item\label{vi} If $P$ is a place of $K(x)$ (resp. $K(y)$) different from
$P_{\alpha_i}$  (resp. $R_{\alpha_i}$) then $P$ is not ramified in
$F$. Therefore the extensions $F/K(x)$ and $F/K(y)$ are weakly ramified so that the extensions $k(x,y)/k(x)$ and $k(x,y)/k(y)$ are also weakly ramified.

\end{enumerate}
\end{lem}
\begin{proof}
Let $\varphi(T)=T^2+T+x/h(x)\in K(x)[T]$. The places $P_0$ and
$P_{\infty}$  of $K(x)$ are zeros of $x/h(x)$ in $K(x)$ so that the
reduction $\varphi(T)\mod P_{\beta}$ with $\beta=0, \infty$ is the
polynomial  $T(T+1)$. Also the reduction $\varphi(T)\mod P_1$ is the
polynomial $T^2+T+1$. Then Kummer's Theorem ( \cite[Theorem
3.3.7]{stichbook}) implies that \ref{i}, \ref{ii} and \ref{iii}
hold.

Let us see \ref{iv}. Let $P=P_{\a_i}$ for $i=1,2$.  Since
$K(x,y)/K(x)$ is an Artin-Schreier extension of degree $2$ and for
$u=0$ we have that
\[m_P=\nu_P\left(u^2+u+\frac{x}{x^2+x+1}\right)=-1\not\equiv0 \mod
2,\] the theory of Artin-Schreier extensions ( \cite[Proposition
3.7.8]{stichbook}) tell us that $P_{\a_i}$ is totally ramified in
$F$ and
 if $Q_i$ is the place of $F$ lying over $P_{\a_i}$ then
$d(Q_i|P_{\a_i})=(2-1)(-m_P+1)=2$. Since $Q_i$ is a pole (of order $2$) of $y^2+y$ in $F$, then $Q_i$ is also a simple pole of $y$ in $F$ so that $Q_i\cap K(y)=R_{\infty}$ which completes the proof of \ref{iv}

Let us see now \ref{v}. Let $Q_i$ be a place of $F$ lying over
$R_{\alpha_i}$ for $i=1,2$. Then $\nu_{Q_i}(y^2+y+1)=e(Q_i|R_{\a_i})$. Let $S_i=Q_i\cap K(x)$. Since
\[\dfrac{(x+1)^2}{h(x)}=y^2+y+1,\] we have
\[e(Q_i|P_{\a_i})=e(Q_i|S_i)\nu_{P_i}((x+1)^2/h(x))=2e(Q_i|S_i)\nu_{S_i}((x+1)/h(x)),\]
so that $e(Q_i|P_{\a_i})=2$, $e(Q_i|S_i)=1$ and $\nu_{S_i}((x+1)/h(x))=1$. Since $h(x)$ is a polynomial, we must have that $S_i=P_1$. Thus we have proved that each $R_{\a_i}$ is totally ramified in $F$ and each $Q_i$ lies over $P_1$ in $K(x)$. In particular $\nu_{Q_i}(x+1)=1$, i.e. $x+1$ is a prime element for $Q_i$. We also have that
\[(x+1)^2+(x+1)\left(\frac{y^2+y+1}{y^2+y}\right)+\frac{y^2+y+1}{y^2+y}=0,\]
so that
\begin{equation}\label{e2}
\phi(T)=T^2+\left(\frac{y^2+y+1}{y^2+y}\right)T+\frac{y^2+y+1}{y^2+y}
\end{equation}
is the minimal polynomial of $x+1$ over $K(y)$ because it is irreducible over $K(y)$ by Eisenstein criterion (\cite[Proposition 3.1.15]{stichbook}) using any of the places $R_{\a_i}$. Then   \cite[Proposition 3.5.12]{stichbook} tell us that
\[d(Q_i|R_{\alpha_i})=\nu_{Q_i}(\phi^{\prime}(x+1))=\nu_{Q_i}\left(\frac{y^2+y+1}{y^2+y}\right)=2,\]  which finishes the proof of \ref{v}.

Finally, let us see \ref{vi}. From \ref{ii} and the theory of
Artin-Schreier extensions we see that the places $P_{\a_i}$ for
$i=1,2$ are the only ones ramified in $F$. Now consider the
polynomial $\phi(T)$ given in $(\ref{e2})$, which is the minimal
polynomial of $x+1$ over $K(y)$. We have that $K(x+1,y)=F$ and also
for any place $P$ of $K(y)$ different from $R_{\b}$ for $\b=0,1$,
the polynomial $\phi(T)$ is integral over $\mathcal{O}_P$, the
valuation ring corresponding to $P$. Let $Q$ be a place of $F$ lying
over $P \not = R_{\beta}$ with $\beta\in \{0,1,\a_1,\a_2\}$, then by
  \cite[Theorem 3.5.10]{stichbook}, we have
\[d(Q|P)\leq \nu_{Q}(\varphi^{\prime}(x+1))=\nu_{Q}\left(\frac{y^2+y+1}{y^2+y}\right)=e(Q|R_{\beta})\nu_{R_{\beta}}\left(\frac{y^2+y+1}{y^2+y}\right)=0,\]
so that $P$ is unramified in $F$. Finally if either $P=R_0$ or $R_1$ and $Q$ is a place of $F$ lying over $P$, then $\nu_Q(y^2+y)=e(Q|P)$. Let $S=Q\cap K(x)$, then
\[1\leq e(Q|P)=e(Q|S)\nu_S(x/h(x)),\]
which implies that $S$ is either $P_0$ or $P_{\infty}$. By \ref{i} and \ref{ii} we have that $e(Q|P)=1$ and we are done.
\end{proof}

\begin{figure}[h!t]\begin{center}

\begin{tikzpicture}[scale=0.7]
\draw[line width=0.5 pt](-16,0)--(-15,2)--(-14,0); %% primer grà fico
  \draw[white, fill=white](-16,0) circle (0.3 cm);
      \draw[white, fill=white](-15,2) circle (0.3 cm);
    \draw[white, fill=white](-14,0) circle (0.3 cm);

\node at (-16,0)  { \footnotesize{$K(x)$}  };%
\node at (-14,0)  {\footnotesize{$K(y)$}};
\node at (-15,2)  {\footnotesize{$F$}};

\draw[line width=0.5 pt](-12,0)--(-11,2)--(-10,0); %% primer grà fico
  \draw[white, fill=white](-12,0) circle (0.3 cm);
    \draw[white, fill=white](-10,0) circle (0.3 cm);

\node at (-12,0)  { \footnotesize{$P_0$}  };%
\node at (-10,0)  {\footnotesize{$R_0$}};

\coordinate (P) at (-11.7,1.3);
\node[rotate=60] (N) at (P) { \tiny{$e=1$}}; %

\coordinate (P) at (-10.3,1.3);
\node[rotate=-63] (N) at (P) { \tiny{$e=1$}}; %

\draw[line width=0.5 pt](-8,0)--(-7,2)--(-6,0); %% segundo gráfico
  \draw[white, fill=white](-8,0) circle (0.3 cm);
    \draw[white, fill=white](-6,0) circle (0.3 cm);

\node at (-8,0)  {\footnotesize{$P_0$}};%
\node at (-6,0)  {\footnotesize{$R_1$}};

\coordinate (P) at (-7.7,1.3);
\node[rotate=60] (N) at (P) {\tiny{$e=1$}}; %

\coordinate (P) at (-6.3,1.3);
\node[rotate=-63] (N) at (P) { \tiny{$e=1$}}; %

\draw[line width=0.5 pt](-4,0)--(-3,2)--(-2,0); %% tercer gráfico
  \draw[white, fill=white](-4,0) circle (0.3 cm);
    \draw[white, fill=white](-2,0) circle (0.3 cm);

\node at (-4,0)  {\footnotesize{$P_{\infty}$}};%
\node at (-2,0)  {\footnotesize{$R_0$}};

\coordinate (P) at (-3.7,1.3);
\node[rotate=60] (N) at (P) { \tiny{$e=1$}}; %

\coordinate (P) at (-2.3,1.3);
\node[rotate=-63] (N) at (P) { \tiny{$e=1$}}; %

\draw[line width=0.5 pt](0,0)--(1,2)--(2,0); %% cuarto gráfico
  \draw[white, fill=white](0,0) circle (0.3 cm);
    \draw[white, fill=white](2,0) circle (0.3 cm);

\node at (0,0)  {\footnotesize{$P_{\infty}$}};%
\node at (2,0)  {\footnotesize{$R_1$}};

\coordinate (P) at (0.3,1.3);
\node[rotate=60] (N) at (P) {\tiny{$e=1$}}; %

\coordinate (P) at (1.7,1.3);
\node[rotate=-63] (N) at (P) { \tiny{$e=1$}}; %

\draw[line width=0.5 pt](-12,-4)--(-11,-2)--(-10,-4); %% primer grà fico de abajo
  \draw[white, fill=white](-12,-4) circle (0.3 cm);
    \draw[white, fill=white](-10,-4) circle (0.3 cm);

\node at (-12,-4)  {\footnotesize{$P_1$}};%
\node at (-10,-4)  {\footnotesize{$R_{\alpha_1}$}};

\coordinate (P) at (-11.7,-2.7);
\node[rotate=60] (N) at (P) {\tiny{$e=1$}}; %

\coordinate (P) at (-10.3,-2.7);
\node[rotate=-63] (N) at (P) {\tiny{$d=e=2$}}; %

\draw[line width=0.5 pt](-8,-4)--(-7,-2)--(-6,-4); %% segundo gráfico de abajo
  \draw[white, fill=white](-8,-4) circle (0.3 cm);
    \draw[white, fill=white](-6,-4) circle (0.3 cm);

\node at (-8,-4)  {\footnotesize{$P_1$}};%
\node at (-6,-4)  {\footnotesize{$R_{\alpha_2}$}};

\coordinate (P) at (-7.7,-2.7);
\node[rotate=60] (N) at (P) { \tiny{$e=1$}}; %

\coordinate (P) at (-6.3,-2.7);
\node[rotate=-63] (N) at (P) { \tiny{$d=e=2$}}; %

\draw[line width=0.5 pt](-4,-4)--(-3,-2)--(-2,-4); %% tercer gráfico de abajo
  \draw[white, fill=white](-4,-4) circle (0.3 cm);
    \draw[white, fill=white](-2,-4) circle (0.3 cm);

\node at (-4,-4)  {\footnotesize{$P_{\alpha_1}$}};%
\node at (-2,-4)  {\footnotesize{$R_{\infty}$}};

\coordinate (P) at (-3.7,-2.7);
\node[rotate=63] (N) at (P) {\tiny{$d=e=2$}}; %

\coordinate (P) at (-2.3,-2.7);
\node[rotate=-63] (N) at (P) {\tiny{$e=1$}}; %

\draw[line width=0.5 pt](0,-4)--(1,-2)--(2,-4); %% cuarto gráfico de abajo
  \draw[white, fill=white](0,-4) circle (0.3 cm);
    \draw[white, fill=white](2,-4) circle (0.3 cm);

\node at (0,-4)  {\footnotesize{$P_{\alpha_2}$}};%
\node at (2,-4)  {\footnotesize{$R_{\infty}$}};

\coordinate (P) at (0.3,-2.7);
\node[rotate=63] (N) at (P) {\tiny{$d=e=2$}}; %

\coordinate (P) at (1.7,-2.7);
\node[rotate=-63] (N) at (P) {\tiny{$e=1$}}; %
    \end{tikzpicture}

   \caption{Ramification in  $F/K(x)$ and  $F/K(y)$ }\label{figu1}
 \end{center}\end{figure}
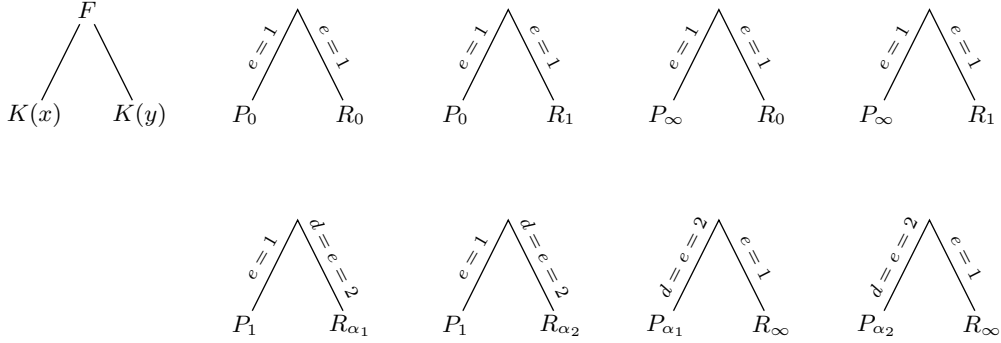

As a direct consequence of the above lemma we have:
\begin{propo}\label{tp} The ramification locus $R(\mathcal{H})$ of the tower   $\mathcal{H}$ is finite. More precisely  \[R(\mathcal{H})\subseteq\{P_0,\,P_1,\,P_{\a_1}, \,P_{\a_2},\,P_{\infty}\}, \]
with the notation used in Lemma \ref{lema3.1}.
\end{propo}

Now we can  state and prove the main result of this section.
\begin{teo}\label{finitegenus} The tower $\mathcal{H}$ is a weakly ramified tower over $\f_{2^s}$ and  its genus $\gamma(\mathcal{H})$ satisfies the estimate
\[\gamma(\mathcal{H})\leq 4,\]
for any positive integer $s$.

\end{teo}

\begin{proof}
Let $s\in\mathbb{N}$ and let $F=k(x,y)$ be the basic function field
associated to $\mathcal{H}$ where $k=\f_{2^s}$. From (vi) of Lemma
\ref{lema3.1} we have that $F/k(x)$ and $F/k(y)$ are weakly ramified
extensions and from Proposition \ref{tp} we also have that the
ramification locus of $\mathcal{H}$ is finite. The conclusions follow
immediately from Corollary \ref{wram} and Proposition \ref{p1}.
\end{proof}

\section{The splitting rate of the tower $\mathcal{H}$}\label{splitting}

Throughout this section $k=\f_{2^s}$ will be a finite field with $2^s$ elements
and $\tr$ denotes the trace map from  $\f_{2^s}$ to $\f_{2}$.

\subsection {\textbf{The odd case.}}
In this subsection we prove one of our main results, namely that the tower $\mathcal{H}$
 over $\f_{2^s}$ has zero splitting rate for every odd integer $s$. We begin with the following technical lemma.

\begin{lem}\label{lema4.1}Let $\t,\b \in k$  such that
$\t^2+\t=\frac{\b}{\b^2+\b+1}.$ Then
\[\tr\left(\frac{\theta}{\theta^2+\theta+1}\right)\neq \tr\left(\frac{\theta+1}{\theta^2+\theta+1}\right).\]
\end{lem}
\begin{proof}   Suppose that
\[\tr\left(\frac{\theta}{\theta^2+\theta+1}\right)= \tr\left(\frac{\theta+1}{\theta^2+\theta+1}\right)\] then
\begin{equation}\label{trace0}
\tr\left(\frac{1}{\theta^2+\theta+1}\right)=0.
\end{equation}
On the other hand, by hypothesis   \[\theta^2+\theta=\frac{\b}{\b^2+\b+1}\]
then
\[\frac{1}{\theta^2+\theta+1}=\frac{\b^2+\b+1}{\b^2+1}=1+\frac{\b}{\b+1}+\left(\frac{\b}{\b+1}\right)^2.\]
Finally, since $\tr(1)=1$ and $\tr(\a)=\tr(\a^2)$ for all $\a \in k$, we have
\[\tr\left(\frac{1}{\theta^2+\theta+1}\right)=\tr(1)+\tr\left(\frac{\b}{\b+1}\right)+\tr\left(\left(\frac{\b}{\b+1}\right)^2\right)=1,\] contradicting \ref{trace0}.
\end{proof}

Now we are in a position to state and prove one of our main results for the odd case.

\begin{teo}\label{teo4.2} Consider the tower $\mathcal{H}=\{F_i\}_{i\geq0}$ over $k=\f_{2^s}$ with $s$ odd. For every $i\geq 1$, the number of rational places $N(F_i)$ of  $F_i$ is \[N(F_i)=2(|S|+1),\]
where \[S=\left\{\b \in k: \tr\left(\frac{\b}{{\b}^2+\b+1}\right)=0\right\}.\]
\end{teo}
\begin{proof}
Let $P$ be a rational place of $F_i$. Since $s$ is odd, we have that $\beta^2+\beta+1\neq 0$ for any $\b\in k$. Then the reduction modulo $P$ of the polynomial
\[\phi=T^2+T+\frac{x_i}{x_i^2+x_i+1}\in F_i[T],\]
is the polynomial
\[\phi_{\beta}=T^2+T+\frac{\beta}{\beta^2+\beta+1}\in k[T].\]
for some $\beta\in k$. On the other hand since
\[\tr\left(\frac{\beta}{\beta^2+\beta+1}\right)=0
\quad\text{if and only if}\quad \frac{\beta}{\beta^2+\beta+1}=\theta^2+\theta,\]
for some $\theta\in k$, it is clear that $\phi_{\beta}$ is irreducible over $k$ if and only if
\[\tr\left(\frac{\beta}{\beta^2+\beta+1}\right)=1.\]
Let $P_{\b}$ be the only zero of $x_0+\b$ in $F_0=k(x_0)$ and  $P_{\infty}$ be the only pole of $x_0$ in $F_0$. Then $\phi \mod P_{\b}=\phi_{\b}$ for $\beta\in k$ and $\phi \mod P_{\infty}=\phi_{0}$. If $\b\notin S\cup\{\infty\}$ we have that $\phi_{\b}$ is irreducible over $k$ and, by Kummer's Theorem (\cite[Theorem 3.3.7]{stichbook}), we see that  $P_{\b}$ is inert in $F_1$, thus there is only one place of $F_1$ lying over $P_{\b}$ and is of degree  $2$ for $\b\notin S\cup\{\infty\}$.  Now let $\b\in S\cup \{\infty\} $. In this case $\phi_{\b}$ splits into two different linear factors over $k$. Clearly if $\t\in k$ is a root of $\phi_{\b}$ then we have that
 $\phi_{\b}(T)=(T+\t)(T+(\t+1))$. Again by Kummer's Theorem, there are exactly two rational places  $Q_{\t}$, $Q_{\t+1}$ of $F_1=F_0(x_1)$ lying over $P_{\b}$. Then
 \[N(F_1)=2(|S|+1).\]
Kummer's Theorem also tell us that $x_1+\t \in Q_{\t}$ and
$x_1+(\t+1) \in Q_{\t+1}$ so the residual classes $x_1(Q_{\t})$ and
$x_1(Q_{\t+1})$ are $\t$ and $\t+1$ respectively.

 Let us consider now the rational places of $F_2$. Each such a place is lying over a rational place of the form above considered $Q_{\t}$ and $Q_{\t+1}$ for some $\t\in k$, so that the reduction of $\phi$ modulo these places are the polynomials $\phi_{\t}$ and $\phi_{\t+1}$.  Now Lemma \ref{lema4.1} implies  that the polynomial
\[\phi_{\t}(T)=T^2+T+\frac{\t}{{\t}^2+\t+1},\]
splits completely over $k$ (resp. is irreducible over $k$) if and only if  the polynomial
\[\phi_{\t+1}(T)=T^2+T+\frac{\t+1}{{\t}^2+\t+1}\] is irreducible over $k$ (resp. splits completely over $k$). Hence $Q_{\t}$ splits completely (resp. remains inert) if and only if  $Q_{\t+1}$ remains inert (resp. splits completely). Therefore, for each pair $Q_{\t}$, $Q_{\t+1}$ we have, again by Kummer's Theorem, that one of them remains inert, say $Q_{\t}$, and there are two rational places of $F_2$  over the other one, $Q_{\t+1}$. Thus,
\[N(F_2)=2(|S|+1).\]
In particular, we have proved that each rational place of $F_2$ lies over a rational place of $F_1$ of the form $Q_{\t}$ for some $\t\in k$, which splits completely in $F_2$ into two rational places of $F_2$ of the form $Q_{\g}$ and $Q_{\g+1}$ with residual classes $x_2(Q_\g)=\g$ and $x_2(Q_{\g+1})=\g+1$ by Kummer's Theorem.

By inductive hypothesis, $N(F_i)=2(|S|+1)$ and each rational place
of $F_{i}$ lies over a rational place of $F_{i-1}$ of the form
$Q_{\t}$ for some $\t\in k$, which splits completely into two
rational places of $F_{i}$ of the form $Q_{\delta}$ and
$Q_{\delta+1}$ with residual classes $x_i(Q_\delta)=\delta$ and
$x_i(Q_{\delta +1})=\delta+1$. Again we see that the reduction of
$\phi$ modulo these two places are the polynomials $\phi_{\delta}$
and $\phi_{\delta+1}$ so that by Lemma \ref{lema4.1} the polynomial
\[\phi_{\delta}(T)=T^2+T+\frac{\delta}{{\delta}^2+\delta+1},\]
splits completely over $k$ (resp. is irreducible over $k$) if and only if  the polynomial
\[\phi_{\delta+1}(T)=T^2+T+\frac{\delta+1}{{\delta}^2+\delta+1}\] is irreducible over $k$ (resp. splits completely over $k$). Therefore we have that either $Q_{\delta}$ or $Q_{\delta+1}$ splits completely in $F_{i+1}$ while the other one remains inert in $F_{i+1}$. Thus \[N(F_{i+1})=2(|S|+1)\] and we are done.
\end{proof}

As an immediate consequence of the above result we have:
\begin{teo}\label{badodd} The tower $\mathcal{H}$  is asymptotically bad over $\f_{2^s}$ for any odd integer $s$.

\end{teo}

\begin{proof} From Theorem \ref{teo4.2} we have that \[\nu(\mathcal{H})=\lim_{i\rightarrow \infty}\frac{N(F_i)}{[F_i:F_0]}=\lim_{i\rightarrow \infty}\frac{2(|S|+1)}{2^i}=0.\]
\end{proof}

\subsection {\textbf{The even case.}}

\medskip

We  show now that  the tower $\mathcal{H}$ has positive splitting rate   over $\f_{2^s}$ for $s$ even so  that $\mathcal{H}$ is asymptotically good in this case.  We treat first the case $s=2$ giving in  Theorem \ref{goodf4} an explicit lower bound for the limit of the tower $\mathcal{H}$ over $\f_4$.

Throughout this subsection we  consider $\f_4:=\{0,1,\a,\a+1\}$  with  $\a^2+\a+1=0$  and  we denote by $\tr$ the trace map from $\f_4$ to $\f_2$.

\begin{propo}\label{propoinicial}
Let $F$ be a function field over $\f_4$ such that $\f_4$ is its full
field of constants and let $x\in F\setminus \f_4$. Let $F'=F(y)$
where $y$ satisfies \eqref{ec1}, i.e.
\[y^2+y=f(x):=\frac{x}{x^2+x+1}.\] Then $F'/F$ is an Artin-Schreier
extension of degree $2$  where any zero $P_{0}$ and any  pole $P_{\infty}$ of $x$ in $F$ respectively, split completely in $F^{\prime}/F$ into a zero $Q_0$ of $y$ and a zero $Q_1$ of $y+1$. Also any  zero $P_1$ of $x+1$ in $F$ splits completely in $F'/F$ into a zero $Q_{\alpha}$ of $y+\alpha$ and a zero $Q_{\alpha+1}$ of $y+\alpha+1$ in $F'$ (see Figure \ref{figu2} below).  Moreover
\[\nu_{Q_i}(y+i)=\nu_{P_0}(x)\qquad \text{ and }\qquad
\nu_{Q_i}(y+i)=-\nu_{P_\infty}(x),\] where $i=0,1$ and
\[\nu_{Q_\b}(y+\b)=2\nu_{P_1}(x+1),\]
where $\b=\a$ or $\b=\a+1$.
\end{propo}

\begin{figure}[h!t]\begin{center}
\begin{tikzpicture}[scale=0.7]

\draw[line width=0.5 pt](-3,0)--(-3,2); %
\node at (-3,-0.3)  {\footnotesize{$F$}  };%
 \node at (-3,2.3)  {\footnotesize{$F'$}};

\draw[-][line width=0.5 pt](0,0)--(-1,2);
\draw[-][line width=0.5 pt](0,0)--(1,2);
\node at (-1.1,2.3)  { \footnotesize{$Q_0$}  };
\node at (0.1,-0.3)  {\footnotesize{$P_0$}};
\node at (1.1,2.3)  {\footnotesize{$Q_1$}};

\coordinate (P) at (-0.7,1);
 \node[rotate=0] (N) at (P){ \tiny{$1$}};

\coordinate (P) at (0.7,1);
\node[rotate=0] (N) at (P){ \tiny{$1$}};

\draw[-][line width=0.5 pt](4,0)--(3,2);
\draw[-][line width=0.5 pt](4,0)--(5,2);
\node at (2.9,2.3)  { \footnotesize{$Q_0$}  };
\node at (4.2,-0.3)  {\footnotesize{$P_{\infty}$}};
\node at (5.1,2.3)  {\footnotesize{$Q_1$}};

\coordinate (P) at (3.3,1);
\node[rotate=0] (N) at (P){ \tiny{$1$}};

\coordinate (P) at (4.7,1);
\node[rotate=0] (N) at (P){ \tiny{$1$}};

\draw[-][line width=0.5 pt](8,0)--(9,2);%
\draw[-][line width=0.5 pt](8,0)--(7,2);%
\node at (9,2.3)  {\footnotesize{$Q_{\a+1}$}};%
\node at (8,-0.3) {\footnotesize{$P_1$}}; %
\node at (7.1,2.3){\footnotesize{$Q_{\a}$}};

\coordinate (P) at (8.3,1);
\node[rotate=0] (N) at (P){ \tiny{$1$}};

\coordinate (P) at (7.7,1);%
\node[rotate=0] (N) at (P){ \tiny{$1$}};%
\end{tikzpicture}
\caption{Decomposition of some zeros and poles}\label{figu2}
\end{center}
\end{figure}
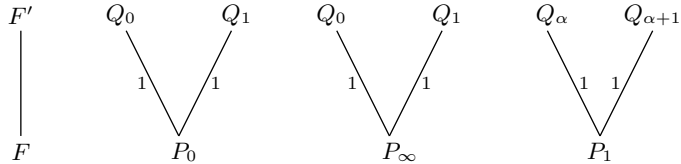

\begin{proof}
Consider the polynomial
 \[\varphi(T)=T^2+T+f(x)\in F[T].\] First notice that each zero or pole  $P$ of $x$ in $F$ is a zero of $f(x)$ in $F$. Then  \[\varphi(T)\mod P=T^2+T=T(T+1),\] so Kummer's Theorem tell us that the first two diagrams are correct, i.e.  there is a zero $Q_0$ of $y$ in $F'$ and a zero $Q_1$ of $y+1$ in $F'$ lying over $P$.  Therefore \[\nu_{Q_i}(y+i)=\nu_{P_0}(x)\qquad \text{ and }\qquad \nu_{Q_i}(y+i)=-\nu_{P_\infty}(x).\]
  Let $P_1$ be a zero of $x+1$ in $F$. Then $\nu_{P_1}(x+1)>0$ and $\nu_{P_1}(x)=0$ so that $\nu_{P_1}(f(x))=0$ and the residual class $x(P_1)=1$. Thus
 \[\varphi(T)\mod P_1=T^2+T+1=(T+\a)(T+\a+1) ,\] and  again Kummer's Theorem shows that the last diagram is correct.  Now, by rewriting \eqref{ec1} as
 \[y^2+y+1=\frac{(x+1)^2}{x^2+x+1},\]  we see that if $Q_\b|P_1$ then  $\nu_{Q_{\b}}(y+\b)=2\nu_{P_1}(x+1)$ as desired.
 \end{proof}

 \begin{propo}\label{reductor}  Under the conditions of Proposition \ref{propoinicial}, let us consider a zero $P_\beta$ of $x+\b$ in $F$.
 \end{propo}

 \begin{enumerate}[(i)]
 \item\label{red1}  Suppose that there exists an element $u \in F$ such that \[\nu_{P_\b}(f(x)+u^2+u)=-1,\]
 then  $P_\b$ is totally ramified in $F'/F$.  The only place of $F'$ lying over $P_{\beta}$ is a pole $Q_{\infty} $ of $y$ in $F'$ and \[\nu_{Q_\ii}(y)=-\nu_{P_\b}(x+\b).\]
 \item\label{red2} Suppose that $P_{\beta}$ is rational and that there exists an element $u \in F$ such that
 \[\nu_{P_\b}(f(x)+u^2+u)\geq0 \quad \text{and}\quad \tr((f(x)+u^2+u)(P_{\b}))=0.\]
 Then  $P_\b$ splits completely in $F'/F$ into two poles of $y$ in $F'$ and  \[2\nu_{Q_\ii}(y)=-\nu_{P_\b}(x+\b)\]
where $Q_{\ii}$ denotes any of these two poles (see Figure \ref{figu3} below).
\end{enumerate}

\begin{figure}[h!t]\begin{center}
\begin{tikzpicture}[scale=0.7]

\draw[line width=0.5 pt](8,0)--(8,2); %
\node at (8,-0.3)  {\footnotesize{$F$}  }; %
\node at (8,2.3){\footnotesize{$F'$}};

\draw[-][line width=0.5 pt](10,0)--(10,2); \node at (10,2.3)  {
\footnotesize{$Q_{\infty}$}  }; \node at (10.2,-0.3)
{\footnotesize{$P_{\beta}$}};

\coordinate (P) at (10.2,1); \node[rotate=0] (N) at (P){
\tiny{$2$}};

\draw[-][line width=0.5 pt](13,0)--(12,2); \draw[-][line width=0.5
pt](13,0)--(14,2); \node at (12.1,2.3)  { \footnotesize{$Q_{\ii}$}
}; \node at (13.1,-0.3)  {\footnotesize{$P_\b$}}; \node at
(14.1,2.3)  {\footnotesize{$Q_{\ii}$}};

\coordinate (P) at (12.3,1); \node[rotate=0] (N) at (P){
\tiny{$1$}};

\coordinate (P) at (13.7,1); \node[rotate=0] (N) at (P){
\tiny{$1$}};
\end{tikzpicture}
\caption{The two possible decompositions of $P_\b$}\label{figu3}
\end{center}\end{figure}
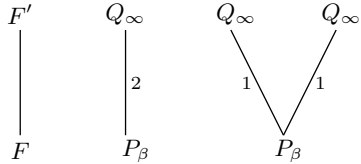

\begin{proof}  The first part of \eqref{red1} is  just \cite[Proposition 3.7.8]{stichbook}.
It is clear from the equation \eqref{ec1} defining the extension $F'/F$ that $\nu_{Q_\ii}(y)=-\nu_{P_\b}(x+\b)$.

 Let us see (ii). From   \cite[Proposition 3.7.8]{stichbook} and its proof, we have that $P_{\beta}$ is unramified in $F'$ and also $F^{\prime}=F(y+u)$ where
\[\varphi(T)= T^2+T+f(x)+u^2+u,\] is the minimal polynomial of $y+u$ over $F$. Then
the reduction of $\varphi(T)$ modulo $P_{\b}$ splits into linear factors over $\f_4$ because $\tr((f(x)+u^2+u)(P_{\b}))=0$ and thus  we can conclude, by Kummer's Theorem, that $P_{\b}$ splits completely in $F^{\prime}/F$. From equation \eqref{ec1} we see at once that $2\nu_{Q_\ii}(y)=-\nu_{P_\b}(x+\b)$.
\end{proof}

The construction of elements satisfying (i) or (ii) in the above proposition will be a crucial technical point in our proof of the existence of a rational place which splits completely in the tower $\mathcal{H}$. At this point we find convenient to introduce the following definition:
\begin{defi}\label{astype}
 Let $F'/F$ be an Artin-Schreier extension defined by \eqref{ec1} as in Proposition  \ref{propoinicial}. Let $P$ be a place of $F$. An element $u\in F$ is  called an Artin-Schreier element of type 1 for $P$ if
\[\nu_{P}(f(x)+u^2+u)=-m,\]
for some odd positive integer $m$. Suppose now that $P$ is rational. An element $u\in F$ is called an Artin-Schreier element of type 2 for $P$ if
\[\nu_{P}(f(x)+u^2+u)\geq0 \quad \text{and}\quad \tr((f(x)+u^2+u)(P))=0.\]

\end{defi}

\begin{remark}
The arguments given in Proposition \ref{reductor} show that a place $P$ of $F$ is totally ramified in $F'/F$ if an  Artin-Schreier element of type 1 for $P$ is found, and if $P$ is rational, $P$ splits completely in $F'/F$ if an Artin-Schreier element of type 2 for $P$ is found.
\end{remark}

We have seen  in Proposition \ref{propoinicial} that the places $P_i$ with $i=0,1,\infty$ splits completely
in $F'/F$. The goal in this section is to prove that  each zero
$Q_{\b}$ of $x_i+\b$ in $F_i$ with $\b\in \{\a,\a+1$\} splits completely in
$F_{i+1}/F_i$ for $i\geq 4$ whenever the place $Q_{\b}\cap F_2$ is a pole
of $x_2$ in $F_2$.  For these purposes, we will need to prove some results of technical nature.
\begin{lem}\label{techlem}Let  us consider the equation
\[y^2+y=f(x):=\frac{x}{x^2+x+1},\]  as in Proposition \ref{propoinicial}. Then
\[f(y)+\left(\frac{y+1}{x+1}\right)^2+\frac{y+1}{x+1}=y+\frac{1}{x^2+x+1}+\frac{1}{x+1}\cdot\]
\end{lem}
\begin{proof}

\begin{align*}
f(y)+\left(\frac{y+1}{x+1}\right)^2+\frac{y+1}{x+1}&= \frac{y}{y^2+y+1}+\left(\frac{y+1}{x+1}\right)^2+\frac{y+1}{x+1} \\
           &= \frac{y}{f(x)+1}+\left(\frac{y+1}{x+1}\right)^2+\frac{y+1}{x+1} \\
&= y\frac{x^2+x+1}{(x+1)^2}+\left(\frac{y+1}{x+1}\right)^2+\frac{y}{x+1}+\frac{1}{x+1} \\
&= \frac{y(x^2+1)+yx+y^2+1+y(x+1)}{(x+1)^2}+\frac{1}{x+1}\\
&= \frac{y(x+1)^2}{(x+1)^2}+\frac{yx+y^2+1+y(x+1)}{(x+1)^2}+\frac{1}{x+1} \\
           &= y+\frac{1}{x^2+x+1}+\frac{1}{x+1}
\end{align*}
\end{proof}

%%%%%%%%%%%%%%%%%%%%%%%%%%%%%%%%%%%%%%%%%%%%%%%%%%%%%%%%%%%%%%%%%%%%%%%%%%%%%%%%%%%%%%%%%%%%%%%%%%%%%%%%%%% SEGUIR HASTA %%%%%%%%%%%%%%%

%%%%%%%%%%%%%%%%%%%%%%%%%%%%%%%%%%%%%%%%%%%%%%%%%%%%%%%%%%%%%%%%%%%%% LEMA 4.8. Generalizado

 From now on we will use the following notation: let $i\geq 0$. A zero of $x_i$ (resp. $x_i+1$) in $F_i$ will be denoted as $Q^i_0$ (resp. $Q^i_1$) and a pole of $x_i$ in $F_i$ will be denoted as $Q^i_{\ii}$. A zero of $x_i+\beta$ in $F_i$ will be denoted as $Q^i_{\beta}$ for $\beta\in\{\alpha,\alpha+1 \}$ and this will be the meaning of any symbol of the form $Q^i_{\gamma}$ when a Greek letter such as $\gamma$ is used as a subindex.  Also from now on all the considered places lie over $Q^0_1$ (the only zero of $x_0+1$ in the rational function field $F_0$ of $\mathcal{H}$). We will write in many occasions $P\subset Q$ when a place $Q$ lies over a place $P$.

We state and prove now three more technical results  (Lemmas \ref{leg4.8}, \ref{leg4.9} and \ref{lemadeunosg}). In the three  of them  we will assume that the following condition hold:

\begin{ramif}\label{ramifcond}
Let $k\geq 0$ and consider the function fields $F_k\subset F_{k+1}\subset F_{k+2}$ of the tower $\mathcal{H}$. For the places  $Q_1^k\subset Q_{\b}^{k+1}$ one and only one of these two conditions holds:
\begin{enumerate}[(R1)]
\item either $\nu_{Q_{1}^{k}}(x_k+1)=1$ and  $Q_{\b}^{k+1}$ is
totally ramified in $F_{k+2}/F_{k+1}$ (so that there is only one pole $Q^{k+2}_{\ii}$ of $x_{k+2}$ in $F_{k+2}$ lying over $Q^{k+1}_{\b}$)\label{tr}, or
\item $\nu_{Q_{1}^{k}}(x_k+1)=2$ and  $Q_{\b}^{k+1}$
splits completely in $F_{k+2}/F_{k+1}$  (so that there are exactly two poles $Q^{k+2}_{\ii}$ of $x_{k+2}$ in $F_{k+2}$ lying over $Q^{k+1}_{\b}$).\label{sc}
\end{enumerate}

\end{ramif}

\begin{lem}\label{leg4.8}  For $k\geq 0$ and $i\geq k+4$ let us consider the subsequence $\{F_j\}_{j=k}^{i-1}$  of the tower $\mathcal{H}$ an also the following sequence of places:
\[Q^{k}_{1}\subset Q^{k+1}_{\b}\subset Q^{k+2}_{\ii}\subset Q^{k+3}_{0}\subset Q^{k+4}_{0}\subset \cdots\subset Q^{i-1}_0, \]
where we are having only the places $Q^j_0$ for $k+3\leq j\leq i-1$. Then

\begin{enumerate}[(i)]
\item\label{4.8.1} $\nu_{Q^{k+3}_0}(x_{k+3}x_{k+2})=\nu_{Q^{j}_{0}}(x_{j}/x_{j-1})=0$  for all $j=k+4,\dots,i-1$  and $i>k+4$.
\item\label{4.8.2} $x_{k+3}x_{k+2}(Q^{k+3}_0)=(x_{j}/x_{j-1})(Q^{j}_0)=1$ for all $j=k+4,\dots,i-1$  and $i>k+4$.
\item\label{4.8.3} For all $j=k+3,\dots,i-1$
\[\nu_{Q^{j}_0}\left(\frac{1}{x_j}+x_{k+2}\right)\geq0 \quad \text{and} \quad \left(\frac{1}{x_j}+x_{k+2}\right)(Q^{j}_{0})=0.\]
\item\label{4.8.4} $Q^{i-1}_0|Q^{k+2}_{\ii}$ is not ramified in $F_{i-1}/F_{k+2}$ and  $\nu_{Q^{i-1}_0}(\d)=-1$  where $i\geq k+4$ and \[\d:=x_{k+2}+\frac{x_{k+1}+1}{x_{k}+1}\cdot\]
\item\label{4.8.5} $(x_{k+2}+(x_{k+1}\d)^2+x_{k+1}\d)(Q_{\ii}^{k+2})=\b.$

\end{enumerate}
\end{lem}

\begin{proof} By Proposition \ref{propoinicial} we know
that $\nu_{Q_{\b}^{k+1}}(x_{k+1}+\b)=2\nu_{Q_{1}^{k}}(x_{k}+1)$ so  that
\[2\nu_{Q_{\ii}^{k+2}}(x_{k+2})=-e(Q_{\ii}^{k+2}|Q_{\b}^{k+1})\nu_{Q_{\b}^{k+1}}(x_{k+1}+\b)=-2e(Q_{\ii}^{k+2}|Q_{\b}^{k+1})\nu_{Q_{1}^{k}}(x_{k}+1).\]
 Since the ramification condition holds, we have  that \[e(Q_{\ii}^{k+2}|Q_{\b}^{k+1})\nu_{Q_{1}^{k}}(x_{k}+1)=2,\] thus
\[\nu_{Q_{\ii}^{k+2}}(x_{k+2})=-e(Q_{\ii}^{k+2}|Q_{\b}^{k+1})\nu_{Q_{1}^{k}}(x_{k}+1)=-2.\]
 From Proposition \ref{propoinicial}  we see that
$Q_0^{j}|Q_{\ii}^{k+2}$ is  unramified in $F_j/F_{k+2}$ and
\begin{equation}\label{doblecero}
\nu_{Q_{0}^{j}}(x_{j})=-\nu_{Q_{\ii}^{k+2}}(x_{k+2})=2,
\end{equation}
 for all
$j=k+3,\dots,i-1$.  Consequently, for $i>k+4$ and  $k+4\leq j\leq i-1$  we have that
\[\nu_{Q_0^{k+3}}(x_{k+3}x_{k+2})
=2-2=0=\nu_{Q_0^{j}}(x_{j})-\nu_{Q_0^{j-1}}(x_{j-1})=\nu_{Q_0^{j}}\left(\frac{x_j}{x_{j-1}}\right).\]
On the other hand, since
\begin{equation}\label{eculeg4.8}
(x_{k+3}+1)x_{k+3}x_{k+2}=\frac{x_{k+2}^2}{x_{k+2}^2+x_{k+2}+1}
\quad \text{ and } \quad
\frac{(x_{j}+1)x_{j}}{x_{j-1}}=\frac{1}{x_{j-1}^2+x_{j-1}+1},
\end{equation}
it follows that
\[ (x_{k+3}x_{k+2})(Q_0^{k+3})=1\quad\text{ and}\quad\frac{x_{j}}{x_{j-1}}(Q_0^{j})=\frac{1}{x_{j-1}^2+x_{j-1}+1}(Q_0^{j-1})=1,\]
 for $i$ and $j$ as above. Let us see now that $\nu_{Q_0}^{i-1}(\d)=-1$.  By Lemma~\ref{techlem} we have that
\begin{align*}
\delta^2+\delta
 &=x_{k+2}^2+x_{k+2}+\left(\frac{x_{k+1}+1}{x_{k}+1}\right)^2+\frac{x_{k+1}+1}{x_{k}+1}\\
\\
&=f(x_{k+1})+\left(\frac{x_{k+1}+1}{x_{k}+1}\right)^2+\frac{x_{k+1}+1}{x_{k}+1}\\
\\
&=x_{k+1}+\frac{1}{x_k^2+x_k+1}+\frac{1}{x_k+1}. \\
\end{align*}
Since $Q_0^{i-1}\cap F_{k+1}=Q_{\beta}^{k+1}$ and $\nu_{Q_{\beta}^{k+1}}(x_{k+1})=0$ we have
\[\nu_{Q_0}^{i-1}(x_{k+1})=\nu_{Q_0}^{i-1}\left(\frac{1}{x_k^2+x_k+1}\right)=0,\]
and
\[\nu_{Q_0}^{i-1}\left(\frac{1}{x_k+1}\right)=-e(Q_{\ii}^{k+2}|Q_{\b}^{k+1})\nu_{Q_{1}^{k}}(x_{k}+1)=-2.\]
Then
\[\nu_{Q_0^{i-1}}(\delta^2+\delta)=\nu_{Q_0^{i-1}}\left(x_{k+1}+\frac{1}{x_k^2+x_k+1}+\frac{1}{x_k+1}\right)=-2,\]
so that $\nu_{Q_0^{i-1}}(\delta)=-1$. This finishes the proof of \eqref{4.8.1}, \eqref{4.8.2} and \eqref{4.8.4}.

 Now we see \eqref{4.8.3} and we proceed by induction on $j$. For $j=k+3$, using the
first identity in \eqref{eculeg4.8}, we have that
\begin{align*}
\frac{1}{x_{k+3}}+x_{k+2}&=\frac{(x_{k+3}+1)x_{k+2}}{(x_{k+3}+1)x_{k+3}x_{k+2}}+x_{k+2}\\
&=(x_{k+3}+1)\frac{x_{k+2}^2+x_{k+2}+1}{x_{k+2}}+x_{k+2}\\
&=(x_{k+3}+1)\left(\frac{1}{x_{k+2}}+1+x_{k+2}\right)+x_{k+2}\\
&=\frac{x_{k+3}}{x_{k+2}}+x_{k+3}+x_{k+3}x_{k+2}+1+\frac{1}{x_{k+2}}.
\end{align*}
Then
\[\nu_{Q_0^{k+3}}\left(\frac{1}{x_{k+3}}+x_{k+2}\right)\geq 0,\]
and
\begin{align*}
\left(\frac{1}{x_{k+3}}+x_{k+2}\right)(Q_0^{k+3})
&=\frac{x_{k+3}}{x_{k+2}}(Q_0^{k+3})+x_{k+3}(Q_0^{k+3})+(x_{k+3}x_{k+2})(Q_0^{k+3})+\\
& \quad +1+\frac{1}{x_{k+2}}(Q_0^{k+3})=0+0+1+1+0=0.
\end{align*}

 Assume now that the result is valid for $j-1\geq k+3$. By the second
identity in \eqref{eculeg4.8} we obtain
\begin{align*}
\frac{1}{x_j}+x_{k+2}&=  \frac{(x_j+1)}{x_{j-1}}\frac{x_{j-1}}{(x_j+1)x_j}+x_{k+2}\\
&= \frac{(x_j+1)}{x_{j-1}}(x_{j-1}^2+x_{j-1}+1)+x_{k+2}\\
&= (x_j+1)\left(\frac{1}{x_{j-1}}+1+x_{j-1}\right)+x_{k+2}\\
&=\frac{x_{j}}{x_{j-1}}+x_j+x_jx_{j-1}+x_{j-1}+1+\frac{1}{x_{j-1}}+x_{k+2}.
\end{align*}

Then
\[\nu_{Q_0^{j}}\left(\frac{1}{x_j}+x_{k+2}\right)\geq 0,\]
 and
\begin{align*}
\left(\frac{1}{x_j}+x_{k+2}\right)(Q_0^{j})
&=\frac{x_{j}}{x_{j-1}}(Q_0^{j})+x_j(Q_0^{j})+(x_jx_{j-1})(Q_0^{j})+x_{j-1}(Q_0^{j})+1+\\
&\quad +\left(\frac{1}{x_{j-1}}+x_{k+2}\right)(Q_0^{j})=1+0+0+0+1+0=0.
\end{align*}

It remains to prove \eqref{4.8.5}. For the proof of this item we will use an identity which is a consequence of some tedious manipulations,
so for the sake of simplicity we find convenient to write $x=x_k$, $y=x_{k+1}$ and $z=x_{k+2}$. We have now that
$\d=z+\frac{y+1}{x+1}$ and $x,y,z$ satisfy the following relations:
\begin{align}
z^2&=z+f(y), \label{ec4a}\\
y^2f(y)&=1+y+\frac{1}{y^2+y+1},\label{ec4b}\\
\frac{1}{y^2+y+1}+\frac{1}{(x+1)^2}+\frac{1}{x+1}&=\frac{x^2+x+1}{(x+1)^2}+\frac{x}{(x+1)^2}=1.\label{ec4c}
\end{align}
Then
\begin{align*}
z+(y\d)^2+y\d &=z+\left[\left(z+\frac{y+1}{x+1}\right)y\right]^2+\left(z+\frac{y+1}{x+1}\right)y\\
&= z+z^2y^2+zy+\frac{y^2(y+1)^2}{(x+1)^2}+\frac{y(y+1)}{(x+1)} \\
           &= z+(z+f(y))y^2+zy+\frac{y^2(y+1)^2}{(x+1)^2}+\frac{y(y+1)}{(x+1)} \quad (\text{by } \eqref{ec4a}) \\
           &= z+zy^2+f(y)y^2+zy +\frac{y^2(y+1)^2}{(x+1)^2}+\frac{y(y+1)}{(x+1)}\\
           &= z(y^2+y+1)+1+y+\frac{1}{y^2+y+1}+\frac{y^2(y+1)^2}{(x+1)^2}+\\
           &\quad +\frac{y(y+1)}{(x+1)} \quad (\text{by } \eqref{ec4b}) \\
           &= z(y^2+y+1)+1+y+\frac{1}{y^2+y+1}+\frac{y^4+y^2+1+1}{(x+1)^2}+\\
           &\quad +\frac{y^2 + y+1 +1}{(x+1)}\\
           &= z(y^2+y+1)+1+y+\frac{(y^2+y+1)^2}{(x+1)^2}+\\
           &\quad+\frac{y^2+y+1}{(x+1)}+1 \qquad (\text{by } \eqref{ec4c}) \\
           &= z\frac{(x+1)^2}{x^2+x+1}+y+\frac{(x+1)^2}{(x^2+x+1)^2}+\frac{x+1}{x^2+x+1}\cdot\\
\end{align*}
Therefore
\begin{align*}
x_{k+2}+(x_{k+1}\d)^2+x_{k+1}\d &=
x_{k+2}\frac{(x_k+1)^2}{x_k^2+x_k+1}+x_{k+1}+\frac{(x_k+1)^2}{(x_k^2+x_k+1)^2}+\\
&\quad+\frac{x_k+1}{x_k^2+x_k+1}.
\end{align*}
Since $\nu_{Q^{k+2}_{\infty}}(x_{k+1})=0$ and the other summands
 in the right hand side of the previous equality have positive
valuations, we have that
\begin{equation}\label{deltasquare}
\nu_{Q^{k+2}_{\ii}}(x_{k+2}+(x_{k+1}\d)^2+x_{k+1}\d )=0,
\end{equation}
 and
\[(x_{k+2}+(x_{k+1}\d)^2+x_{k+1}\d)(Q^{k+2}_{\ii})=x_{k+1}(Q^{k+2}_\infty)=x_{k+1}(Q^{k+1}_\b)=\b.\]

\end{proof}

\begin{lem}\label{leg4.9}  For $k\geq 0$ and $i\geq k+4$ let us consider the subsequence $\{F_j\}_{j=k}^{i+2}$ of the tower
$\mathcal{H}$ and the sequence of places
\[Q^{k}_{1}\subset Q^{k+1}_{\b}\subset Q^{k+2}_{\ii}\subset Q^{k+3}_{0}\subset \cdots\subset Q^{i-1}_0 \subset Q^i_1\subset Q^{i+1}_{\theta}\subset Q^{i+2}_{\ii},  \]
where we are having only the places $Q^j_0$ for $k+3\leq j\leq i-1$.
 Then $Q_{\theta}^{i+1}$ splits completely in $F_{i+2}/F_{i+1}$  and the following holds:
\begin{enumerate}[(i)]
\item\label{i1} $\nu_{Q_1^{i}}(x_{i}+1)= \nu_{Q_1^{i}}({x_{i-1}})
=2$, $\nu_{Q_1^{i}}(\frac{x_{i}+1}{x_{i-1}})=0$  and  $\frac{x_{i}+1}{x_{i-1}}(Q_1^{i})=1$.
 \item\label{i2} $\nu_{Q_{\ii}^{i+2}}(x_{i}+1)=-\nu_{Q_{\ii}^{i+2}}(x_{i+2})=2$,  $\nu_{Q_{\ii}^{i+2}}(x_{i+1})=0$ and  $\nu_{Q_{\ii}^{i+2}}(\d^{\prime})=-1$ where \[\d^{\prime}=x_{i+2}+\frac{x_{i+1}+1}{x_{i}+1}.\]
\item\label{i3} $(x_{i+2}+(x_{i+1} \d^{\prime})^2+x_{i+1}\d^{\prime})(Q_{\ii}^{i+2})=\theta.$
\end{enumerate}
\end{lem}

\begin{proof}Let us prove  \eqref{i1}.  From \eqref{doblecero} and Proposition \ref{propoinicial} we have that $Q_0^{i-1}$ is a zero (of order $2$) of $x_{i-1}$ and splits completely in $F_{i}/F_{i-1}$. By hypothesis $Q_1^i|Q_0^{i-1}$, then $Q_1^{i}$ is a  zero (of order $2$) of $x_{i}+1$  so that $\nu_{Q_1^{i}}(\frac{x_{i}+1}{x_{i-1}})=0$. On the other hand since \[\frac{x_{i}(x_{i}+1)}{x_{i-1}}=\frac{1}{x_{i-1}^2+x_{i-1}+1},\]
 we see at once that  $\frac{x_{i}+1}{x_{i-1}}(Q_1^{i})=1$  and we are done with item  \eqref{i1}. Now we show  that
\[u:=\frac{x_{i+1}+1}{x_{i}+1}+x_{k+1}\d\] is  an Artin-Schreier  element of type 2 for
$Q_{\theta}^{i+1}$ where $\delta$ is as in (iv) of Lemma \ref{leg4.8}. Indeed, using Lemma~\ref{techlem} we have
\begin{align*}
f(x_{i+1})+u^2+u &= \frac{x_{i+1}}{x_{i+1}^2+x_{i+1}+1}+\left(\frac{x_{i+1}+1}{x_{i}+1}\right)^2+\frac{x_{i+1}+1}{x_{i}+1}+(x_{k+1}\d)^2\\
&\quad + (x_{k+1}\d) \\
&= x_{i+1}+\frac{1}{x_{i}^2+x_{i}+1}+\frac{1}{x_{i}+1}+(\delta x_{k+1})^2+\delta x_{k+1} \\
&= x_{i+1}+\frac{1}{x_{i}^2+x_{i}+1}+\frac{x_{i}}{f(x_{i-1})}+(\delta x_{k+1})^2+\delta x_{k+1} \\
&= x_{i+1}+\frac{1}{x_{i}^2+x_{i}+1}+x_{i}\left(x_{i-1}+1+\frac{1}{x_{i-1}}\right)+(\delta x_{k+1})^2+\\
&\quad +\delta x_{k+1} \\
&= x_{i+1}+\frac{1}{x_{i}^2+x_{i}+1}+x_{i}x_{i-1}+x_{i}+\frac{x_{i}}{x_{i-1}}+(\delta x_{k+1})^2+\\
&\quad + \delta x_{k+1} \\
&= x_{i+1}+\frac{1}{x_{i}^2+x_{i}+1}+x_{i}x_{i-1}+x_{i}+\frac{x_{i}+1}{x_{i-1}}+\\
&\quad +\left(\frac{1}{x_{i-1}}+x_{k+2}\right)+(x_{k+2}+(\delta x_{k+1})^2+\delta x_{k+1}).
\end{align*}
Now from Proposition \ref{propoinicial},  \eqref{4.8.3} and \eqref{4.8.5} of Lemma \ref{leg4.8} and \eqref{deltasquare}, we have that $Q_{\theta}^{i+1}$ is not a pole of any of the terms in the above last equality. Therefore
\[\nu_{Q_{\theta}^{i+1}}(f(x_{i+1})+u^2+u) \geq 0.\]
 Also the residual class $(f(x_{i+1})+u^2+u)(Q_{\theta}^{i+1})$ is
\begin{align*}
  & x_{i+1}(Q_{\theta}^{i+1})+\frac{1}{x_{i}^2+x_{i}+1}(Q_{\theta}^{i+1})+(x_{i}x_{i-1})(Q_{\theta}^{i+1})+x_{i}(Q_{\theta}^{i+1})+\frac{x_{i}+1}{x_{i-1}}(Q_{\theta}^{i+1})\\
    &+\left(\frac{1}{x_{i-1}}+x_{k+2}\right)(Q_{\theta}^{i+1}) +(x_{k+2}+(\delta x_{k+1})^2+\delta x_{k+1})(Q_{\theta}^{i+1})=\theta+1+0+\\
      &+1+1+0+\b=\theta+1+\b.
\end{align*}
Since $\theta+1+\b$  is equal to $0$ or $1$, then
\[\tr\left((f(x_{i+1})+u^2+u)(Q_{\theta}^{i+1})\right)=0,\]
and, as we have seen in Remark \ref{reductor}, this implies that  the place
$Q_{\theta}^{i+1}$ splits completely in $F_{i+2}/F_{i+1}$. Let us prove now \eqref{i2}. We have
\[x_{i+1}^2+x_{i+1}+1=\frac{(x_i+1)^2}{x_i^2+x_i+1},\]
so that from \eqref{i1} we see that
$\nu_{Q_{\theta}^{i+1}}(x_{i+1}+\theta)=4$. Since
$e(Q_{\ii}^{i+2}|Q_{\theta}^{i+1})=1$ we deduce
\[2\nu_{Q_{\ii}^{i+2}}(x_{i+2})=\nu_{Q_{\ii}^{i+2}}(x_{i+2}^2+x_{i+2})=\nu_{Q_{\theta}^{i+1}}\left(\frac{x_{i+1}}{x_{i+1}^2+x_{i+1}+1}\right)=-4,\]
and then $\nu_{Q_{\ii}^{i+2}}(x_{i+2})=-2$. Since $e(Q_{\ii}^{i+2}|Q_1^{i})=1$, from \eqref{i1} we also have that $\nu_{Q_{\ii}^{i+2}}(x_i+1)=2$. Clearly $\nu_{Q_{\ii}^{i+2}}(x_{i+1})=0$ and the first part of \eqref{i2} follows.  Let us see now
that $\nu_{Q_{\ii}^{i+2}}(\d^{\prime})=-1$. In fact,
\begin{align*}
{\delta^{\prime}}^2+\delta^{\prime}&=x_{i+2}^2+x_{i+2}+\left(\frac{x_{i+1}+1}{x_{i}+1}\right)^2+\frac{x_{i+1}+1}{x_{i}+1}\\
                                   &=f(x_{i+1})+\left(\frac{x_{i+1}+1}{x_{i}+1}\right)^2+\frac{x_{i+1}+1}{x_{i}+1}\\
                                   &=x_{i+1}+\frac{1}{x_{i}^2+x_{i}+1}+\frac{1}{x_{i}+1}, \\
\end{align*} by Lemma~\ref{techlem}.  Then \[\nu_{Q_{\ii}^{i+2}}({\delta^{\prime}}^2+\delta^{\prime})=\nu_{Q_{\ii}^{i+2}}\left(x_{i+1}+\frac{1}{x_{i}^2+x_{i}+1}+\frac{1}{x_{i}+1}\right)=-2,\]
so that $\nu_{Q_{\ii}^{i+2}}({\delta^{\prime}})=-1$ and we are done with item \eqref{i2}.
Finally \eqref{i3} follows by noticing that  the same argument given in the proof of \eqref{4.8.5} of Theorem \ref{leg4.8} applies in this case because we have that the ramification condition (R\ref{sc})  holds with $k=i$ and $\b=\t$ and  the element $\delta'$ is just $\delta$ with $k=i$.
\end{proof}

\begin{lem}\label{lemadeunosg}For $k\geq 0$ let us consider the subsequence $\{F_j\}_{j=k}^{k+5}$ of the tower $\mathcal{H}$ and the sequence of places:
\[Q^{k}_{1}\subset Q^{k+1}_{\b}\subset Q^{k+2}_{\ii}\subset Q^{k+3}_{1}\subset Q^{k+4}_{\theta}\subset Q^{k+5}_{\ii}. \]
Then the place $Q_{\theta}^{k+4}$ splits completely in $F_{k+5}/F_{k+4}$. Furthermore, we have that the following properties hold:
\begin{enumerate}[(i)]
\item\label{item4.6.1} $\nu_{Q_1^{k+3}}(x_{k+3}+1)=-\nu_{Q_1^{k+3}}({x_{k+2}})=2$, $\nu_{Q_1^{k+3}}((x_{k+3}+1)x_{k+2})=0$  and \\ $((x_{k+3}+1)x_{k+2})(Q_1^{k+3})=1$.
%\item The places $Q_{\theta}^{k+4}$ splits completely in $F_{k+4}/F_{k+5}$.
\item\label{item4.6.2} $\nu_{Q_{\ii}^{k+5}}(x_{k+3}+1)=-\nu_{Q_{\ii}^{k+5}}(x_{k+5})=2$ and  $\nu_{Q_{\ii}^{k+5}}(x_{k+4})=0$.
\item\label{item4.6.3} $\nu_{Q_{\ii}^{k+5}}(\d^{''})=-1$ where \[\d^{''}=x_{k+5}+\frac{x_{k+4}+1}{x_{k+3}+1}.\]
\item\label{item4.6.4} $(x_{k+5}+(x_{k+4} \d^{''})^2+x_{k+4}\d^{''})(Q_{\ii}^{k+5})=\theta.$
\end{enumerate}
\end{lem}

\begin{proof} Let us prove \eqref{item4.6.1}.  As we have seen at the beginning of the proof of Lemma \ref{leg4.8},  the ramification conditions (R\ref{tr}) or (R\ref{sc}) imply that  $Q_{\ii}^{k+2}$ is a pole of order $2$ of $x_{k+2}$ in $F_{k+2}$ and from Proposition \ref{propoinicial} we also have that  $Q_{\ii}^{k+2}$ splits completely in $F_{k+3}/F_{k+2}$.  Then
\[\nu_{Q_1^{k+3}}(x_{k+3}+1)=2=-\nu_{Q_1^{k+3}}({x_{k+2}}),\]
which implies that $\nu_{Q_1^{k+3}}((x_{k+3}+1)x_{k+2})=0$. Since
\[x_{k+3}((x_{k+3}+1)x_{k+2})=\frac{x_{k+2}^2}{x_{k+2}^2+x_{k+2}+1}=\frac{1}{1+\frac{1}{x_{k+2}}+\frac{1}{x_{k+2}^2}},\]
we see that
\begin{align*}
(x_{k+3}((x_{k+3}+1)x_{k+2}))(Q_1^{k+3}) &= x_{k+3}(Q_1^{k+3})((x_{k+3}+1)x_{k+2})(Q_1^{k+3})=\\ &= \frac{1}{1+\frac{1}{x_{k+2}}+\frac{1}{x_{k+2}^2}}(Q_{\ii}^{k+2})=1.
\end{align*}
Since $x_{k+3}(Q_1^{k+3})=1$ we conclude that $((x_{k+3}+1)x_{k+2})(Q_1^{k+3})=1$ and this finishes the proof of \ref{item4.6.1}. We show now that
\[u:=\frac{x_{k+4}+1}{x_{k+3}+1}+x_{k+1}\delta, \]
where $\delta$ is as in Lemma \ref{leg4.8}, is an Artin-Schreier element of type $2$ for $Q_\theta^{k+4}$. Indeed, using Lemma \ref{techlem} we have
\begin{align*}
f(x_{k+4})+u^2+u &= f(x_{k+4})+\left(\frac{x_{k+4}+1}{x_{k+3}+1}\right)^2+\frac{x_{k+4}+1}{x_{k+3}+1}+(x_{k+1}\delta)^2+x_{k+1}\delta \\
           &= x_{k+4}+\frac{1}{x_{k+3}^2+x_{k+3}+1}+\frac{1}{x_{k+3}+1}+(x_{k+1}\delta)^2+ x_{k+1}\delta \\
           &= x_{k+4}+\frac{1}{x_{k+3}^2+x_{k+3}+1}+x_{k+3}\left(x_{k+2}+1+\frac{1}{x_{k+2}}\right)+\\ &\quad +(x_{k+1}\delta)^2 +x_{k+1}\delta \\
           &= x_{k+4}+\frac{1}{x_{k+3}^2+x_{k+3}+1}+x_{k+3}x_{k+2}+\\ &\quad +x_{k+3}\left(1+\frac{1}{x_{k+2}}\right) +(x_{k+1}\delta)^2+x_{k+1}\delta \\
           &= x_{k+4}+\frac{1}{x_{k+3}^2+x_{k+3}+1}+x_{k+3}\left(1+\frac{1}{x_{k+2}}\right)+x_{k+3}x_{k+2}+\\
           &\qquad +x_{k+2}+x_{k+2}+(x_{k+1}\delta)^2+x_{k+1}\delta \\
           &= x_{k+4}+\frac{1}{x_{k+3}^2+x_{k+3}+1}+x_{k+3}\left(1+\frac{1}{x_{k+2}}\right)+  \\
           &\quad +(x_{k+3}+1)x_{k+2} + x_{k+2}+ (x_{k+1}\delta)^2+x_{k+1}\delta.
\end{align*}
Using \eqref{item4.6.1} and \eqref{deltasquare} it can be easily checked  that $Q_{\theta}^{k+4}$ is not a pole of any of the summands of the above last equality. Thus
\[ \nu_{Q_{\theta}^{k+4}}(f(x_{k+4})+u^2+u)\geq  0,\]
and the residual class $(f(x_{k+4})+u^2+u)(Q_{\theta}^{k+4})$ is
\begin{align*}
 & x_{k+4}(Q_{\theta}^{k+4})+\left(\frac{1}{x_{k+3}^2+x_{k+3}+1}\right)(Q_{\theta}^{k+4}) +\left(x_{k+3}(1+x_{k+2}^{-1})\right)(Q_{\theta}^{k+4})+\\ &+((x_{k+3}+1)x_{k+2})(Q_{\theta}^{k+4})+ (x_{k+2}+(x_{k+1}\delta)^2+x_{k+1}\delta)(Q_{\theta}^{k+4})=\theta+\\
& +1+1+1+\b=\theta+1+\b,
\end{align*}
where we have used \eqref{4.8.5} of Lemma \ref{leg4.8} in the last summand.
 Since $\theta+1+\b$ is equal to $0$ or $1$ then $\tr\left((f(x_{k+4})+u^2+u)(Q_{\theta}^{k+4})\right)=0$ and, as we have seen in Remark \ref{reductor}, this implies that   $Q_{\theta}^{k+4}$ splits completely in $F_{k+5}/F_{k+4}$.  Since
\[x_{k+4}^2+x_{k+4}+1=\frac{(x_{k+3}+1)^2}{x_{k+3}^2+x_{k+3}+1},\]
from \eqref{item4.6.1} we have that $\nu_{Q_{\theta}^{k+4}}(x_{k+4}+\theta)=4$ and since $e(Q_{\ii}^{k+5}|Q_{\theta}^{k+4})=1$ we obtain
\[2\nu_{Q_{\ii}^{k+5}}(x_{k+5})=\nu_{Q_{\ii}^{k+5}}(x_{k+5}^2+x_{k+5})=\nu_{Q_{\theta}^{k+4}}\left(\frac{x_{k+4}}{x_{k+4}^2+x_{k+4}+1}\right)=-4,\]
from which we can conclude that $\nu_{Q_{\ii}^{k+5}}(x_{k+5})=-2$. From \eqref{item4.6.1} we see at once that $\nu_{Q_{\ii}^{k+5}}(x_{k+3}+1)=2$ and thus we are done with the proof of \eqref{item4.6.2}. Let us see now that \eqref{item4.6.3} holds. In fact,
\begin{align*}
{\delta^{\prime}}^2+\delta^{\prime}& =x_{k+5}^2+x_{k+5}+\left(\frac{x_{k+4}+1}{x_{k+3}+1}\right)^2+\frac{x_{k+4}+1}{x_{k+3}+1}\\
&=f(x_{k+4})+\left(\frac{x_{k+4}+1}{x_{k+3}+1}\right)^2+\frac{x_{k+4}+1}{x_{k+3}+1}\\
&=x_{k+4}+\frac{1}{x_{k+3}^2+x_{k+3}+1}+\frac{1}{x_{k+3}+1},
\end{align*}
by Lemma~\ref{techlem}. Therefore $\nu_{Q_{\ii}^5}(\delta^{''})=-1$
because
\[\nu_{Q_{\ii}^{k+5}}({\delta^{''}}^2+\delta^{''})=\nu_{Q_{\ii}^{k+5}}\left(x_{k+4}+\frac{1}{x_{k+3}^2+x_{k+3}+1}+\frac{1}{x_{k+3}+1}\right)=-2.\]

 Finally \eqref{item4.6.4} follows directly from \eqref{4.8.5} of Lemma \ref{leg4.8}, because we have that the ramification condition (R\ref{sc}) holds with $k$ and $\b$ replaced by $k+3$ and $\t$ respectively.
\end{proof}

Now we show that the ramification condition (R\ref{tr}) holds for $k=0$.
\begin{propo}\label{propofinal}Let us consider  the subsequence $\{F_j\}_{j=0}^{2}$  of the tower $\mathcal{H}$ and  the  sequence of places
\[Q_1^0\subset  Q_{\b}^{1} \subset Q_{\infty}^{2}.\]
Then the place $Q_{\b}^{1}$ is totally ramified in $F_2/F_1$ and $\nu_{Q^2_{\infty}}(\delta)=-1$ where \[\delta=x_2+\frac{x_1+1}{x_0+1}.\]
Furthermore, we have that the following properties hold:
\begin{enumerate}[(i)]
\item\label{4.11.1} $\nu_{Q^2_{\infty}}(x_0+1)=-\nu_{Q^2_{\infty}}(x_2)=2$ and $\nu_{Q^2_{\infty}}(x_1)=0$.
\item\label{4.11.2} $(x_2+(x_1 \d)^2+x_1\d)(Q^2_\ii)=\b$ and $\nu_{Q^2_{\infty}}(x_2+(x_1 \d)^2+x_1 \d)=0$.

\end{enumerate}
\end{propo}

\begin{proof}
It is clear that $\nu_{Q^1_\b}(x_0+1)=1$.  If we write $u=\frac{x_1+1}{x_0+1}$ then, by Lemma~\ref{techlem}, we  have \[f(x_1)+u^2+u=x_1+\frac{1}{x_0^2+x_0+1}+\frac{1}{x_0+1}\cdot\]
 On the other hand, since $\delta=x_2+u$, we see that
\[\d^2+\d=x_2^2+x_2+u^2+u=f(x_1)+u^2+u=x_1+\frac{1}{x_0^2+x_0+1}+\frac{1}{x_0+1},\]
and
\[\nu_{Q^2_\ii}(\d^2+\d)=e(Q^2_{\ii}|Q^1_{\b})\nu_{Q^1_\b}\left(x_1+\frac{1}{x_0^2+x_0+1}+\frac{1}{x_0+1}\right)=-e(Q^2_{\ii}|Q^1_{\b}),\]
so that $\nu_{Q^2_\ii}(\d)<0$ and this implies that
$\nu_{Q^2_\ii}(\d)=-1$ and $e(Q^2_{\ii}|Q^1_{\b})=2$. Thus \eqref{4.11.1} is straightforward and since $\nu_{Q^0_1}(x_0+1)=1$ and $Q^1_{\b}$ is totally ramified in $F_2/F_1$, the ramification condition (R\ref{tr}) holds with $k=0$ so that \eqref{4.11.2} follows by taking $k=0$ in \eqref{4.8.5} of Lemma \ref{leg4.8}.
\end{proof}
 We are finally in a position to state and prove the main results  for the even case.

\begin{teo}\label{goodeven}  Let $F_0=\f_4(x_0)\subset F_1\subset F_2$ be the first three steps of the tower $\mathcal{H}$. Any pole  of $x_2$ in $F_2$ lying over the zero of $x_0+1$ in $F_0$ splits completely in $\mathcal{H}$.
\end{teo}

\begin{proof} Let $Q^2_{\ii}$ be a pole of $x_2$ in $F_2$ lying over the zero $Q^0_1$ of $x_0+1$ in $F_0$. Notice that Proposition \ref{propoinicial} tell us that we are in the situation of Proposition \ref{propofinal}, i.e.
\[Q_1^0\subset  Q_{\b}^{1} \subset Q_{\infty}^{2},\]
 where we recall that $Q_{\b}^{1}$ denotes a zero of $x_{1}+\b$ in $F_{1}$ with $\beta \in \{\alpha, \alpha+1\}$. In particular, the ramification condition (R\ref{tr}) holds for $k=0$. Also, from Proposition \ref{propoinicial} we see that $Q^2_{\ii}$ splits into two places of $F_3$, namely:
 \begin{enumerate}[(a)]
    \item a zero $Q^3_1$ of $x_3+1$ in $F_3$ and
    \item a zero $Q^3_0$ of $x_3$ in $F_3$.
 \end{enumerate}
   In case (a) the place $Q^3_1$ splits completely into the zeros $Q^4_{\gamma}$ of $x_4+\gamma $ for $\gamma \in \{\alpha, \alpha+1\}$ and Lemma \ref{lemadeunosg} with $k=0$, tell us that $Q^4_{\gamma}$ splits completely in $F_5/F_4$, for $\gamma \in \{\alpha, \alpha+1\}$, into two poles $Q^5_{\infty}$ of $x_5$ in $F_5$. Furthermore from \eqref{4.11.1} of Proposition \ref{propofinal} we have that
 \[\nu_{Q^3_1}(x_3+1)=\nu_{Q^3_1}(x^2_3+x_3)=\nu_{Q^2_{\ii}}\left(\frac{x_2^2}{x_2^2+x_2+1}\right)=2,\]
 so that the ramification condition (R\ref{sc}) holds for $k=3$. Suppose now that we have the sequence of places
 \[Q^i_1\subset Q^{i+1}_{\gamma}\subset Q^{i+2}_{\infty}\subset Q^{i+3}_1\subset Q^{i+4}_{\gamma}\subset Q^{i+5}_{\infty}, \]
  lying over $Q^3_1$ for some $i\geq 3$. If the ramification condition (R\ref{sc}) holds for $k=i$, then by Lemma \ref{lemadeunosg} we have that $Q^{i+4}_{\gamma}$ splits completely in $F_{i+5}/F_{i+4}$ and $\nu_{Q^{i+3}_1}(x_{i+3}+1)=2$, so that the ramification condition (R\ref{sc}) holds for $k=i+3$. Since the ramification condition (R\ref{sc}) holds for $k=3$, we have by an inductive argument, that the ramification condition (R\ref{sc}) holds for $k=i$ for any $i\geq 3$ in the sequence of places
\begin{equation}\label{sequence1-I}
Q^3_1\subset Q^{4}_{\gamma}\subset Q^{5}_{\infty}\subset Q^6_1\subset Q^{7}_{\gamma}\subset Q^{8}_{\infty}\subset \cdots \subset Q^i_1\subset Q^{i+1}_{\gamma}\subset Q^{i+2}_{\infty},
\end{equation}
where the pattern $Q^j_1\subset Q^{j+1}_{\gamma}\subset Q^{j+2}_{\infty}$ goes one after another, i.e. for  $j=3,6,9,\ldots,i$. Now suppose that we have the sequence of places
\begin{equation}\label{sequence1-II}
Q^j_1\subset Q^{j+1}_{\gamma}\subset Q^{j+2}_{\infty}\subset Q^{j+3}_0\subset Q^{j+4}_0\subset \cdots\subset Q^{i-1}_0\subset Q^{i}_1\subset Q^{i+1}_{\gamma}\subset Q^{i+2}_{\infty},
\end{equation}
lying over $Q^3_1$ and that the ramification condition (R\ref{sc}) holds for $k=j$. Then Lemma \ref{leg4.9} tell us that $Q^{i+1}_{\gamma}$ splits completely in $F_{i+2}/F_{i+1}$ and $\nu_{Q^{i}_{1}}(x_i+1)=2$ so that the ramification condition \ref{sc} holds for $k=i$. Now from Proposition \ref{propoinicial} we see that if we go down in the tower $\mathcal{H}$ from the place $Q^j_1$ in \eqref{sequence1-II}, we will find sequences of the form \eqref{sequence1-I} or \eqref{sequence1-II} one after another until we reach  a chain of places of the form \eqref{sequence1-I} for some index $k\geq 3$, and we have seen that the ramification condition (R\ref{sc}) holds for such an index $k$. Therefore, an inductive argument shows that every time we have a sequence of places of the form
\[Q^j_1\subset Q^{j+1}_{\gamma}\subset Q^{j+2}_{\infty},\]
for  $j\geq 3$ lying over $Q^3_1$, the ramification condition (R\ref{sc}) holds for $k=j$. In particular we have that $Q^i_{\gamma}$ splits completely in $F_{i+1}/F_i$, for $\gamma \in \{\alpha, \alpha+1\}$, into two poles $Q^{i+1}_{\infty}$ of $x_{i+1}$ in $F_{i+1}$ for any $i\geq 3$ whenever $Q^i_{\gamma}$ lies over $Q^3_1$. Since all the other type of places lying over $Q^3_1$ also split completely in each step of the tower,  we conclude that $Q^3_1$ splits completely in $\mathcal{H}$.

Now we consider case (b).  From Proposition \ref{propoinicial} we see that $Q^3_0$ splits completely into a zero $Q^4_0$ of $x_4$ and a zero $Q^4_1$ of $x_4+1$. The same proposition tell us that $Q^4_0$ splits completely into a zero $Q^5_0$ of $x_5$ and a zero $Q^5_1$ of $x_5+1$ and so on. In each of these cases we have the sequence of places
\[Q^3_0\subset Q^4_1\subset Q^{5}_{\gamma}\subset Q^{6}_{\infty},\]
and the sequence
\[Q^3_0\subset Q^4_0\subset \cdots\subset Q^{i-1}_0\subset Q^i_1\subset Q^{i+1}_{\gamma}\subset Q^{i+2}_{\infty},\]
for $i> 4$. In the first case we are in the conditions of Lemma \ref{leg4.9} with $k=0$ and $i=4$. Since the ramification condition (R\ref{tr}) holds for $k=0$ we have from Lemma \ref{leg4.9} that the place $Q^{5}_{\gamma}$ splits completely in $F_6/F_5$ and $\nu_{Q^4_1}(x_4+1)=2$ so that ramification condition (R\ref{sc}) holds for $i=4$. In the second case  we are in the conditions of Lemma \ref{leg4.9} with $k=0$ and $i>4$. Since the ramification condition (R\ref{tr}) holds for $k=0$ we have from Lemma \ref{leg4.9} that $Q^{i+1}_{\gamma}$ splits completely in $F_{i+2}/F_{i+1}$ and $\nu_{Q^i_1}(x_i+1)=2$ so that the ramification condition (R\ref{sc}) also holds for $i>4$. Now  every time we have a sequence of places of the form
\[Q^i_1\subset Q^{i+1}_{\gamma}\subset Q^{i+2}_{\infty},\]
for  $i\geq 4$ lying over $Q^3_0$, from Proposition \ref{propoinicial} we see that if we go down in the tower, we will get a sequence of places of the form
\[Q^j_1\subset Q^{j+1}_{\gamma}\subset Q^{j+2}_{\infty}\subset\cdots\subset Q^{i}_1\subset Q^{i+1}_{\gamma}\subset Q^{i+2}_{\infty},\]
where the ramification condition (R\ref{sc}) holds for $k=j$ and having sequences of the form \eqref{sequence1-I} or \eqref{sequence1-II} in between. The same inductive argument we have already used in the case (a) shows in this case that the ramification condition (R\ref{sc}) holds for $k=i$.  In particular we have that $Q^i_{\gamma}$ splits completely in $F_{i+1}/F_i$, for $\gamma \in \{\alpha, \alpha+1\}$, into two poles $Q^{i+1}_{\infty}$ of $x_{i+1}$ in $F_{i+1}$ for any $i\geq 4$ whenever $Q^i_{\gamma}$ lies over $Q^3_0$. Since any other kind of places lying over $Q^3_0$ also split completely in each step of the tower,  we conclude that $Q^3_0$ splits completely in $\mathcal{H}$. Finally, since $Q^2_{\ii}$ splits completely into the places $Q^3_1$ and $Q^3_0$ in $F_3/F_2$, we have that $Q^2_{\ii}$ splits completely in the tower $\mathcal{H}$.
\end{proof}

\begin{teo}\label{goodf4}
The tower $\mathcal{H}$ is asymptotically good over $\f_{4}$ and with limit
\[\lambda(\mathcal{H})\geq  \frac{1}{8}.\]
\end{teo}
\begin{proof}
From Theorem \ref{finitegenus}  we know that the genus $\gamma(\mathcal{H})$ of $\mathcal{H}$ satisfies the inequality
\[\gamma(\mathcal{H})\leq 4. \]
On the other hand,  from Proposition \ref{propoinicial} we see  that $Q^0_1$ splits into the two places $Q^1_{\alpha}$ and $Q^1_{\alpha+1}$ of $F_1$ and each of them is totally ramified in $F_2/F_1$ into a pole of $x_2$ in $F_2$. Therefore both poles of $x_2$ in $F_2$ are rational places and  we have that these two places  split completely in $\mathcal{H}$ by Theorem \ref{goodeven} so that $N(F_i)\geq 2^{i-1}$. Thus
\[\nu(\mathcal{H})\geq \frac{1}{2}, \]
and the conclusion readily follows.
\end{proof}

\begin{cor} The tower $\mathcal{H}$ is asymptotically good over $\f_{2^s}$ for any $s$ even with limit 
\[\lambda(\mathcal{H})\geq  \frac{1}{8}.\]
\end{cor}
\begin{proof}
The genus of $\mathcal{H}$ is invariant under constant field extensions and Proposition \ref{extensionporconstantes} tell us that the splitting rate is non decreasing under constant field extensions. The conclusion now follows immediately from Theorem \ref{goodf4}.
\end{proof}

\subsection*{Acknowledgements}
Partially supported by CONICET and Proyecto CA\-ID 2011, Nro. 501 201101 00308LI - UNL


\begin{thebibliography}{HD}

\bibitem[BGS]{BGS06}
P. Beelen, A. Garcia, and H. Stichtenoth, {\em Towards a classification of recursive towers of function fields over
  finite fields}, Finite Fields Appl., 12(1), (2006), 56--77.

\bibitem[GS]{galoisclosure}
 A. Garcia and H. Stichtenoth, {\em On the Galois Closure of Towers}, Recent trends in coding theory and its applications,  AMS/IP Stud. Adv. Math., 41, Amer. Math. Soc., Providence, RI, (2007), 83-92.


\bibitem[S]{stichbook}
H.Stichtenoth, {\em Algebraic function fields and codes},  GTM 254, Springer, Berlin, 2nd edition, 2009.



\end{thebibliography}
\end{document}